\documentclass[15pt]{amsart}
\usepackage[all,cmtip]{xy}
\usepackage{amsthm}
\usepackage{xypic}
\usepackage{graphicx}
\usepackage{amscd}
\usepackage{setspace}

\usepackage[colorlinks=false, linkcolor=red]{hyperref}

\newcommand{\Z}{\mathcal{Z}}

\DeclareMathOperator{\Span}{Span}
\DeclareMathOperator{\im}{Im}
\DeclareMathOperator{\pr}{pr}
\DeclareMathOperator{\degr}{deg}
\DeclareMathOperator{\ke}{Ker}
\DeclareMathOperator{\Pic}{Pic}

\newcommand{\M}{\overline{\mathcal{M}}}

\newcommand{\C}{\mathbb{C}}
\newcommand{\U}{\mathcal{U}}
\newcommand{\F}{\mathcal{F}}

\newcommand{\PP}{\mathbb{P}}
\newcommand{\PPT}{\mathbb{T}}

\newcommand{\OO}{\mathcal{O}}

\newtheorem{theorem}{Theorem}[section]
\newtheorem{thm}[theorem]{Theorem}

\newtheorem{lemm}[theorem]{Lemma}
\newtheorem{prop}[theorem]{Proposition}
\newtheorem{propo}[theorem]{Proposition}

\newtheorem{defi}[theorem]{Definition}

\newtheorem{remark}[theorem]{Remark}

\numberwithin{equation}{theorem}
\title []{spaces of conics on low degree complete intersections}

\author {xuanyu pan}
\address{Department of Mathematics, Washington University in St.Louis, St.Louis, MO 63130}
\email{pan@math.wustl.edu}
\date{\today}
\begin{document}

\maketitle

\begin{abstract}
Let $X$ be a smooth complete intersection contained in $\PP^n_{\C}$ and of low degree. We consider conics contained in $X$ and passing through two general points of $X$. We show that the moduli space of these conics is a smooth complete intersection in a projective space. The main ingredients of the proof are a criterion for characterizing when a smooth projective variety is a complete intersection in a projective space, the Grothendieck-Riemann-Roch theorem, and the geometry of spaces of conics.
\end{abstract}
\tableofcontents

%\chapter{}
\section {Introduction}

In this paper,  we work over the complex numbers $\mathbb{C}$. Let $X$ be a smooth projective variety in $\PP^n_{\mathbb{C}}$, and let $\M_{0,2}(X,2)$ be the Kontsevich moduli space parametrizing the data $(C,f,x_1,x_2)$ of
\begin{enumerate}
\item  a proper, connected, at-worst-nodal, arithmetic genus $0$ curve $C$,
\item an ordered collection $x_1$ and $x_2$ of distinct smooth points of $C$,
\item and a morphism $f:C\rightarrow X$ whose image has degree $2$
\end{enumerate}
such that $(C,f,x_1,x_2)$ has only finitely many automorphisms. %Here $\alpha$ is the homology class of a line. For simplicity, we denote by $\M_{0,2}(X,2)$ the moduli space $\M_{0,2}(X,2\alpha)$.
For the space $\M_{0,2}(X,2)$, we have an evaluation morphism (cf. \cite{FP})
 \begin{equation} \label{fiberf}
\text{ev}:\M_{0,2}(X,2) \rightarrow X\times X, \ \ (C,f,x_1,x_2)
\mapsto (f(x_1),f(x_2)).
\end{equation}

In the following,  we say that a complete intersection of codimension $k$ is of type $(c_1,c_2,\ldots,c_k)$ if it is defined by $k$ homogeneous polynomials $F_j$ of degree $c_j$ for $j=1,\ldots,k$. We use the notation $(c_1,c_2,\ldots,c_k)-(c_1,c_2,\ldots,c_j)$ to represent the $(k-j)$-tuple $(c_{j+1},c_{j+2},\ldots,c_k)$.\\
%In the paper, we assume the following hypotheses.
\\
\textbf{Conditions.}
Throughout this paper, we always assume that
\begin{itemize}
\item $X$ is a smooth complete intersection in $\PP^n_{\C}$ of type $(d_1,\ldots,d_c)$,
with $c\leq n$ and $d_i\geq 2$ such that $n\geq2 \sum\limits_{i=1}^c d_i-c+1$;
\item and $p$ and $q$ are general points of $X$;
\item and $\F$ is the general fiber of the evaluation map ev over $(p,q)\in X\times X$, cf. (\ref{fiberf}).
\end{itemize}
Since $p$ and $q$ are general, the line $\overline{pq}$ is not contained in $X$. The stable maps parametrized by $\F$ are immersions and their images in $X$ are reducible conics passing through $p$ and $q$. Therefore, we hope it would not cause any confusion if we consider $\F$ as the Hilbert scheme parametrizing conics contained in $X$ and passing through $p$ and $q$.

%Consider the general fiber of $ev$ over $(p,q)\in X\times X$ and denote it by $\F$. The fiber $\F$ parametrizes conics contained in $X$ and passing through $p$, $q$.

In topology, a path connected space is simply connected if the space of based paths is path connected. The topological obstruction theory predicts that there exists a section of a Serre fibration if its fibers are simply connected and the base is a CW-complex of dimension at most two. In algebraic geometry, de Jong and Starr \cite{DS} introduce an algebraic and geometric analogue of simple connectedness, namely, rational simple connectedness (see \cite[Theorem 1.1]{DS}).  Rational simple connectedness plays a similar role as simple connectedness in the topological obstruction theory. Namely, the rational simple connectedness of a smooth projective variety $X$ in $\PP^n_{\mathbb{C}}$ implies some arithmetic properties of $X$, such as weak
approximation and the existence of rational points over function fields of surfaces (\cite{DS}, \cite{DS2}, and \cite{BH}). In general, the first step of showing the rational simple connectedness of $X$ is to show the rational connectedness of $\F$. In this paper, we show that the fiber $\F$ is a smooth complete intersection in a projective space, see Theorem \ref{mainthmb}. As a result, it gives rise to an alternative proof of the rational connectedness of $\F$ for a low degree complete intersection $X$, cf. \cite[Lemma 6.5]{DS}.

On the other hand, Qile Chen and Yi Zhu \cite{ZC} recently use the results in this paper and $\mathbb{A}^1$-curves to prove strong approximation for low degree affine complete intersections over function fields, which is considered to be more difficult to show than weak approximation.

The main theorem of this paper is the following.
\begin{thm}\label{mainthmb}
With the conditions as above, the general fiber $\F$ is of the expected dimension $n+1-2\sum\limits_{i=1}^{c} d_i+c.$ Denote by $\varphi$ the morphism
\begin{equation}\label{vmap}
\varphi :\F\rightarrow \mathbb{P}^{n-2}=\mathbb{P}^{n}/\Span(p,q)
\end{equation}
associating to a stable map $[f:C \rightarrow X,p,q]$ $\in \F$ the point \[[\Span(f(C))]\in \mathbb{P}^{n-2}=\mathbb{P}^{n}/\Span(p,q).\]
Then the morphism $\varphi :\F\rightarrow \mathbb{P}^{n-2}$ is an immersion and the general fiber $\F$ is a smooth complete intersection in $\mathbb{P}^{n-2}$ of type\[(1,1,2,2,\dots,d_1-1,d_1-1,d_1; 1,1,2,2,\dots,d_2-1,d_2-1,d_2;\ldots\]
\[;1,1,2,2,\dots,d_c-1,d_c-1,d_c)-(1,1,2).\]

\end{thm}

For a formal definition of $\varphi$, we refer to \cite[Lemma 6.4]{DS}. 

There is also an interesting application of this theorem in enumerative geometry. Namely, if the number of conics contained in $X$ and passing through $p$ and $q$ is finite, then the number is equal to the degree of $\F$ via the immersion $\varphi$. This number can be easily calculated by the theorem. It can also be  calculated by using quantum cohomology, cf. \cite[Corollary in Page 8]{BV}.

The proof of Theorem \ref{mainthmb} will be given at the end of this paper. Here is a rough sketch. In Section 3, we prove that $\F$ can be embedded into a projective space by a canonical map $\varphi$. We also prove that the boundary divisor $\Delta$ of $\F$ is a smooth complete intersection in the projective space under this embedding. In Section 4, we study the geometry of spaces of conics contained in a projective space.

In Section 5, we prove that the total space $\U$ of the universal family $\U\rightarrow \F$ of conics is smooth. In Section 6, we apply the Grothendieck-Riemann-Roch theorem to $\U\rightarrow \F$. We prove an identity relating divisors of $\F$ in Section 6, see Lemma \ref{twotimes}. Then we use the results in Section 4 to deduce an integral version of this identity, see Proposition \ref{asdiv}.

In Section 7, we give a criterion of when a projective variety is a complete intersection in a projective space in Proposition \ref{mainprop1}. Roughly speaking, we show that a smooth projective variety $Y(\subseteq \PP^n)$ which contains a smooth divisor $X$ is a complete intersection in $\PP^n$ if
\begin{itemize}
\item $X$ is a complete intersection in $\PP^n$ of type $(d_1,\ldots,d_c)$,
\item and the divisor $X$ is the intersection of $Y$ and a hypersurface of degree $d$ in $\PP^n$ where $d\in \{d_1,\ldots,d_c\}$.
\end{itemize}
We apply this proposition to the case $(X,Y)=(\Delta, \F)$. Using Proposition \ref{asdiv}, we verify that $(\Delta, \F)$ satisfy the conditions of Proposition \ref{mainprop1}. Therefore, we conclude that $\F$ is a smooth complete intersection in $\PP^{n-2}$.
\\
\\ 

To avoid repetitions, we fix the following notation throughout this paper:\\
%\\
%\\
%\textbf{Notations.} %and the assumptions as follows.
%Let $n,c,d_1,\ldots,d_c$ be natural numbers such that $c\leq n$ and $d_i\geq 2$. Suppose that $X$ is a smooth complete intersection in $\mathbb{P}^n$ of type \[\underline{d}=(d_1, d_2,\dots ,d_c)\] and $\underline{d}\neq (2)$. Obviously, the projective variety $X$ is linearly nondegenerate, i.e., the variety $X$ is not contained in any hyperplane. Suppose that $p$ and $q$ are general points of $X$.
%Since $p$ and $q$ are general points of $X$,
%One can show that the fiber $\F$ is a smooth projective variety if \[n\geq2 \sum\limits_{i=1}^c d_i-c-1,\] cf. \cite[Lemma 6.5]{DS} and Lemma \ref{lemm01} below. %We would like to sketch a proof here for the sake of the completeness.
%\begin{lemm}\cite[Lemma 6.5]{DS}
%the fiber $\F$ is a smooth projective variety of the expected dimension if $n\geq2 \sum\limits_{i=1}^c d_i-c-1$.
%\end{lemm}
\begin{itemize}
\item Denote by $\Delta$ the boundary divisor of $\F$ parametrizing reducible conics.
\item Denote the universal conic bundle by $\pi:\U \rightarrow \F$ and the natural map from $\U$ to $\F$ by $f : \U \rightarrow X.$ The universal sections are denoted by $\sigma_i :\F\rightarrow \U~(i=0,1)$ with $\im(f\circ \sigma_0)=\{p\}$ and $\im(f\circ \sigma_1)=\{q\}.$%\begin{equation} \label{imageconstant}
%\im(f\circ \sigma_0)=\{p\}\text{~ and~} \im(f\circ \sigma_1)=\{q\}.
%\end{equation}

\item Denote by $\Span(C)$ the unique 2-plane such that $C$ is a subscheme of this 2-plane where $C$ is a conic in $\PP^n$.
\item Denote by $\PPT_{X,x}$ the closure of $T_{X,x}$ in $\PP^n$ where $T_{X,x}$ is the tangent space to $X$ at a point $x\in X$. We say that $\PPT_{X,x}$ is the projective tangent space to $X$ at $x$.
\item Denote by $\overline{st}$ (or $\Span(s,t)$) the line passing through two points $s$ and $t$ in $\PP^n$.
\item Denote by $\PP(V)/\PP(W)$ the projective space $\PP(V/W)$ for a flag $(W\subseteq V)$ of a vector space $V$, e.g. $\PP^n/\Span(C)$, $\PP^n/\Span (s,t)$.

\end{itemize}

\section{Preliminaries}\label{s3}

Most results in this section are well known to experts. For the sake of completeness, we sketch some of the proofs for these results. Recall that $X$ is a smooth complete intersection of type $(d_1,\ldots,d_c)$ in $\PP^n$. Assume that $X=\bigcap\limits_{i=1}^c X_i$ where $X_i$ is a hypersurface of degree $d_i$ in $\PP^n$. %The proof is similar to \cite[Lemma 2.1]{Cubic}.
%The following lemma is well known to the experts.
\begin{lemm} \label{cilines}
Let $L_x$ be the union of lines contained in $X$ and passing through $x\in X$. Assume that $X$ is covered by lines. The space $L_x (\subseteq X)$ is a complete intersection in $\PP^n$ of type
\[\left( \begin{array}{c}
1,2,\dots,d_1-1,d_1;\\
 1,2,\ldots,d_2-1,d_2;\\
 \ldots\\
 \ldots\\
 1,2,\dots,d_c-1,d_c
\end{array} \right)\] if $x$ is a general point of $X$. Moreover, the first $"1"$ in the i-th row represents the linear form defining the projective tangent hyperplane $\PPT_{X_i,x}$ to $X_i$ at $x$ and $"d_i"$ in the i-th row represents the polynomial of degree $d_i$ defining $X_i$.
\end{lemm}
The proof of this lemma is based on local calculations. See the proof of \cite[Lemma 2.1]{Cubic} . One can show Lemma \ref{cilines} in the same way.\\

We have the following lemma.
\begin{lemm}\cite[Lemma 5.1]{DS} \label{lemm01}
Assume that $n\geq2 \sum\limits_{i=1}^c d_i-c+1$. The general fiber $\F$ is smooth and of the expected dimension $n+c+1-2\sum \limits_{i=1}^c d_i.$
Moreover, the intersection of $\F$ and the boundary of $\M_{0,2}(X,2) $ is a simple normal crossing divisor $\Delta$.
\end{lemm}

\begin{prop}\label{sm0}
Assume that $n\geq2 \sum\limits_{i=1}^c d_i-c+1$.
\begin{enumerate}
\item The general fiber $\F$ is a smooth variety of the expected dimension \[n+c+1-2\sum \limits_{i=1}^c d_i.\] \\
\item The boundary divisor $\Delta$ is a smooth complete intersection in $\mathbb{P}^n$ of type
\begin{equation}\label{type}
\left( \begin{array}{c}
 1,1,2,2,\dots,d_1-1,d_1-1,d_1;\\
  1,1,2,2,\ldots,d_2-1,d_2-1,d_2;\\
\ldots\\
\ldots\\
1,1,2,2,\dots,d_c-1,d_c-1,d_c.
 \end{array}\right)
 \end{equation}
\end{enumerate}
 \end{prop}
\begin{proof}
The first assertion follows from Lemma \ref{lemm01}. %%and \cite[Corollary 5.11]{DS} %%
The second assertion follows from the fact \cite[Page 83 (2)]{DS2}. We sketch a proof of the second assertion here for the sake of completeness.

Let $L_p$ (resp. $L_q$) be the union of lines contained in $X$ and passing through the point $p\in X$ (resp. $q \in X$) (cf. Lemma \ref{cilines}). Since $p$ (resp. $q$) is a general point of $X$, one can show that $L_p(\subseteq X)$ (resp. $L_q$) is a complete intersection in $\PP^n$, cf. Lemma \ref{cilines}.  A reducible conic $C$ contained in X and passing through $p$ and $q$ is uniquely determined by the node point. Namely, the conic $C$ is the union $\overline{Qp}\cup \overline{Qq}$ where $Q$ is the node point of $C$. On the other hand, the boundary divisor $\Delta$ parametrizes reducible conics contained in $X$ and passing through $p$ and $q$. It follows that $\Delta=L_p \cap L_q$. By Lemma \ref{lemm01}, we conclude that $L_p$ and $L_q$ intersect properly.
%\item $\Delta$ is smooth by the generic smoothness theorem,
%\item $\Delta=L_p \cap L_q$,
%\item $L_p$ and $L_q$ intersect properly.
%\end{itemize}
The second assertion follows.
\end{proof}

\section{An Embedding map and Projective Geometry }
%Let us define a natural map $\varphi$ from $\F$ to $\PP^{n-2}$ first.\\
%\\
%\textbf{The Morphism $\varphi$.} %Suppose that $W$ is a subspace of a vector space $V$. We denote by $\PP(V)/\PP(W)$ the projective space $\PP(V/W)$. Therefore, we can denote by $\mathbb{P}^{n}/\Span(p,q)$ the space

%The space $\mathbb{P}^{n}/\Span(p,q)$ parametrizing 2-planes contained in $\PP^n$ and containing $\overline{pq}$. It is clear that $\mathbb{P}^{n}/\Span(p,q)$ is a projective space of dimension $n-2$. Recall that $p,q$ are two general points of $X$. Therefore, we have a morphism $\varphi$
%\begin{equation}\label{vmap}
%\varphi :\F\rightarrow \mathbb{P}^{n-2}=\mathbb{P}^{n}/\Span(p,q)
%\end{equation}
%which associates to a stable map $[f:C \rightarrow X,p,q]$ $\in \F$ the point $[\Span(f(C))]\in \mathbb{P}^{n-2}=\mathbb{P}^{n}/\Span(p,q)$, cf. \cite[Lemma 6.4]{DS}.

In this section, we show that the morphism $\varphi$ is a closed immersion. Recall that $X$ is a smooth complete intersection $X_1\cap \dots \cap X_c$ in $\mathbb{P}^n$,  where $X_i$ is a hypersurface of $d_i$ in $\mathbb{P}^n$ for $i=1,\ldots,c$.
%Let $\mathbb{P}^{n-2}$ in (\ref{vmap}) be the intersection $\PPT_{X_1,p}\cap \PPT_{X_1,q}$ of the projective tangent hyperplanes to $X_1$ at the points $p$ and $q$ respectively. %The morphism $\varphi$ map a point $[C]\in \F$ to the point \[\varphi([C])=\Span(C)\cap \mathbb{P}^{n-2}.\] In the following, we assume that $p,q$ are general points and \[n\geq2 \sum\limits_{i=1}^c d_i-c+1\]so that Proposition \ref{sm0} holds.
\begin{lemm} \label{var}
With the notation as above, there is a commutative diagram as follows:

\[
\xymatrix{
\U \ar[d]^{\pi} \ar[r]^f &X\ar@{.>}[dr]^{\pr|_{X}} \ar@{.>}[r]^i &\mathbb{P}^{n}-\overline{pq} \ar[d]^{\pr} \\
\F \ar[rr]^{\varphi} & &\mathbb{P}^{n-2}
}
\]
where \begin{enumerate}
        \item the map $\pr$ is the projection from the line $\overline{pq}$ to a projective subspace $\mathbb{P}^{n-2}$,
        \item and the rational map $i$ is the natural rational inclusion,
        \item and $\PP^{n-2}=\PPT_{X_1,p}\cap \PPT_{X_1,q}$ where $X=X_1\cap \dots \cap X_c$.
      \end{enumerate}
\end{lemm}

\begin{proof}
Note that $i$ is defined on $X-\overline{pq}$. Let $u\in \U$ be a point whose image $f(u)$ is not on the line $\overline{pq}$. Therefore, the line $\overline{pq}$ and the point $f(u)$ span a 2-plane $\Span(f(u),\overline{pq})$. Since $\mathbb{P}^{n-2}$ in the diagram above is the intersection $\PPT_{X_1,p}\cap \PPT_{X_1,q}$ of the projective tangent hyperplanes
to $X_1$, the projection map $\pr|_X$ maps $f(u)$ to the point $\Span(f(u),\overline{pq})\cap \mathbb{P}^{n-2}$. On the other hand, the point $$\varphi(\pi(u))\in \PP^{n-2}=\PP^n/\Span(p,q)$$parametrizes the 2-plane $\Span(C)$ where $C$ is a conic parametrized by $\pi(u)$. Therefore, the point $f(u)$ is on $C$ and $\Span(C)=\Span(f(u),\overline{pq})$. In other words, we have $\pr|_X(f(u))=\varphi(\pi(u))$. We have proved that the diagram is commutative.
\end{proof}

\begin{lemm} \label{inj}
The map $\varphi :\F\rightarrow \mathbb{P}^{n-2}=\mathbb{P}^{n}/\Span(p,q)$ is injective.
\end{lemm}

\begin{proof}
Since X is smooth, the hypersurface $X_1$ is smooth at the point $p$. Let $\PPT_{X_1,p}$ be the projective tangent hyperplane to $X_1$ at $p$. It is clear that the complete intersection $X$ is linearly nondegenerate. In particular, the variety $X$ is not contained in $\PPT_{X_1,p}$. Therefore, the line $\overline{pq}$ is transversal to $\PPT_{X_1,p}$ if the point $q$ in $X$ is general. In this case, the multiplicity of the intersection point of $\overline{pq}$ and $X_1$ at $p$ is $1$.

Suppose that $\varphi$ is not injective. By Lemma \ref{var}, there are two distinct conics $C_1$ and $C_2$ contained in $X$ and passing through $p$ and$ q$. These two conics are lying in a 2-plane $P$. Since the line $\overline{pq}$ is transversal to $\PPT_{X_1,p}$, we have that \[\overline{pq}\not\subseteq X_1.\]In particular, the intersection $X_1\cap P$ is a reducible plane curve $D$.  It is clear that $C_1\cup C_2 \subseteq D$. We conclude that the multiplicity of the intersection $D\cap \overline{pq}$ at $p$ is at least $2$ since $C_1$ and $C_2$ pass through $p$. As a result, the multiplicity of the intersection of $\overline{pq}$ and $X_1$ at $p$ is at least $2$ since the intersection $\overline{pq}\cap X_1$ is the intersection $D\cap \overline{pq}$ in $P$. It is a contradiction.
\end{proof}

\begin{prop}\label{propemb}
The map $\varphi :\F\rightarrow \mathbb{P}^{n-2}=\mathbb{P}^{n}/\Span(p,q)$ is a closed immersion.

\end{prop}

\begin{proof}
By Lemma \ref{inj}, it suffices to show that the differential $d \varphi $ on the tangent space $T_c \F$ to $\F$ at $c\in \F$ is injective for every $c\in \F$. Let $C$ be the conic parametrized by the point $c\in \F$.

%Note that we are able to construct a map $\rho_p:\F\rightarrow \mathbb{P}^{n-2}$ (resp. $\rho_q:\F\rightarrow \mathbb{P}^{n-2}$) associating to the point $$c=[C,C\subseteq X,p,q]\in \F$$ the point $[\Span(p,q,T_p C)/\Span(p,q) ]\in\PP^{n-2}=\PP^n/\Span(p,q)$ (resp. \[[\Span(p,q,T_q C)/\Span(p,q) ]\in\PP^{n-2}=\PP^n/\Span(p,q)),\] cf. (\ref{vmap}), where $T_p C$ is the tangent space to $C$ at $p$.
Take a nonzero vector $v \in T_c \F$. It induces a nonzero normal vector field $N_v\in H^0(C,N_{C/X}(-p-q))$ where $N_{C/X}$ is the normal bundle of $C$. This normal vector field corresponds to the first order deformation of the conic $C$ with fixed points $p$ and $q$.

Note that $\U$ is smooth (we postpone the proof in Section \ref{sectionofsm}), see Proposition \ref{sm1}. For a smooth point $s\in C\subseteq \U$, we can lift the vector $v$ (locally on $\U$) to a vector field $w$ (around $s$) of $T_{\U}|_{C}$  where $T_{\U}$ is the tangent bundle of $\U$. Let $proj$ be the natural map $T_{X}|_{C}\rightarrow N_{C/X}$. Note that the map $proj$ induces the map $H^0(proj)$
\[H^0(C,T_{X}|_C)\rightarrow H^0(C,N_{C/X})\]
 where $H^0(C,T_{X}|_C)$ is the space of the first order deformation of the morphism $C\subseteq X$ leaving domain and
target fixed (see \cite[Section 3.4.1]{deform}), $H^0(C,N_{C/X})$ is the space of the first order deformation of the morphism $C\subseteq X$ leaving target fixed (see \cite[Section 3.4.2]{deform} and \cite[Remarks 3.4.10]{deform}). %By \cite[Proposition 3.4.2]{deform} and \cite[Theorem 3.4.8]{deform}, the map $H^0(proj)$ is the natural map between this two deformation spaces. It follows from \cite[Remarks 3.4.10]{deform} that the space $H^0(C,N_{C/X})$ can be considered as the space of the first order deformation of $C$ in $X$.

Note that $df(w)\in H^0(U,T_X|_C)$ for some open neighborhood $U$ of $s$ is the first order deformation of the morphism $f|_U:U\rightarrow X$ leaving domain and target fixed, cf. \cite[Page 5]{CaJ}. It follows from the remarks in the previous paragraph that $H^0(proj)_s\circ df (w_s)$ is the normal vector $N_{v,s}$ of $N_v$ at the point $s$, i.e., \[H^0(proj)_s\circ df (w_s)=N_{v,s},\]since the normal vector field $N_{v}$ is induced by the first order deformation (corresponding to $v$) of the conic $C$, where $w_s$ is the vector field $w$ at $s$. In particular, we have that $d\pi(w_s)=v$ for any smooth point $s\in C$. By Lemma \ref{var}, it follows that \[d \varphi (v)=d\varphi (d (\pi)(w_s))=d(\pr|_X)(df(w_s)).\]

%can be considered as a vector in $T_{X,s}$ via the map .

%Recall that

We claim that the vector $df(w_s)\in T_{X,s}$ points out of the 2-plane $\mathbb{P}^2=\Span(C)$. If the vector $df(w_s)\in T_{X,s}$ is on the 2-plane $\Span(C)$, then $df(w_s)$ and the tangent vector $T_{C,s}$ to $C$ at the point $s$ span the 2-plane $\Span(C)$. In particular, we have $\Span(C)\subseteq \PPT_{X,s}$. Since the point $s$ is a smooth point of $C$, we have $\Span(C)\subseteq \PPT_{X,p}$ by specializing $s$ to $p$. On the other hand, the line $\overline{pq}$ ($\subseteq \Span(C)$) is not contained in $\PPT_{X,p}$ since $p$ and $q$ are two general points of $X$. It is a contradiction. We have proved the claim.

Therefore, the vector $d (\pr|_X)(df(w_s))$ is nonzero. In other words, the differential $d \varphi $ on the tangent space $T_c \F$ to $\F$ at $c$ is injective.
\end{proof}

\begin{lemm}\label{delci}
The boundary divisor $\Delta$ is a complete intersection in $\PP^{n-2}$, with respect to the immersion $\varphi |_{\Delta}:\Delta \rightarrow \mathbb{P}^{n-2}$, and of type \[(2,2,\dots,d_1-1,d_1-1,d_1; \dots;1,1,2,2,\dots,d_c-1,d_c-1,d_c)\] (just take out of the first two "1" from the tuple in Proposition \ref{sm0}).
\end{lemm}

\begin{proof}
Let us briefly recall how to identify $\Delta$ as a complete intersection in $\PP^n$, cf. Proposition \ref{sm0}. The divisor $\Delta$ parametrizes reducible conics contained in $X$ and passing through $p$ and $q$. Let $C$ be a reducible conic parametrized by $\Delta$. We have that $C=l_p\cup l_q$ where $l_p$ (resp. $l_q$) is a line contained in $X$ and passing through p (resp. q).

The lines $l_p$ and $l_q$ intersect at a point $Q$. It is clear that the reducible  conic $C$ is $\overline{Qp}\cup \overline{Qq}$ where $\overline{Qp}$ (resp. $\overline{Qq}$) is the line $l_p$ (resp. $l_q$). In particular, the intersection point $Q$ determines the reducible  conic $C$. Let  $L_p$ (resp. $L_q$) be the union of the lines contained in $X$ and passing through $p$ (resp. $q$). The intersection $L_p\cap L_q$ parametrizes reducible conics passing through $p$ and $q$. By Proposition \ref{sm0} and its proof, we have that
\begin{itemize}
\item $\Delta=L_p\cap L_q$,
\item and $\Delta$ is a smooth complete intersection of type (\ref{type}) in $\PP^n$.
\end{itemize}
By Lemma \ref{cilines}, the first two $"1"$ in the first row of the tuple (\ref{type}) in Proposition \ref{sm0} represent the linear forms defining the projective tangent hyperplanes $\PPT_{X_1,p}$ and $\PPT_{X_1,q}$. Since $\PP^{n-2}$ in \ref{var} is the intersection of $\PPT_{X_1,p}$ and $\PPT_{X_1,q}$, we have proved the lemma.

%we know $\Delta$ is a smooth complete intersection and its defining equations are the union of the defining equations of the complete intersections $L_p$ and $L_q$.

%Furthermore, for a general point $x\in X$, the union $L_x$ of lines on $X$ passing through $x$ is a complete intersection in $\PP^n$ of type
%\[(1,2,\ldots, d_1;1,2,\ldots,d_2;\ldots; 1,2,\ldots, d_c)\]
%where $d_i$ is the degree of the hypersurface $X_i$. In particular, there are $c$ equations of degree $1$ contributing to define $L_x$ and these linear equations define the projective tangent hyperplanes of $X_1,\ldots,X_c$ at $x$. Therefore, the divisor $\Delta$ is a smooth complete intersection of type \[(2,2,\dots,d_1-1,d_1-1,d_1;1,1,2,2,\ldots,d_2-1,d_2-1,d_2; \dots;1,1,2,2,\dots,d_c-1,d_c-1,d_c)\] in $\PP^{n-2}=\PP T_{X,p}\cap \PP T_{X,q}$ via $\varphi$ by Lemma \ref{var}.
\end{proof}

%The following lemmas will only be used until Section 6.

\section{The Geometry of Parameter Spaces of Conics} \label{sectionofconics}
In the following, a reduced conic refers to a smooth conic or a reducible conic. The main result of this section is Lemma \ref{ktimes}. The presentations of the proofs in this section are suggested by the anonymous referee. Denote by $Sch/\C$ the category of $\C$-schemes.

\begin{defi}
Let $M$ be a functor from the category $Sch/\C$ to the category of sets as follows:
\[M:Sch/\C\rightarrow Sets\]
associates to a scheme $S$ the set
$M(S)= \{ \pi_1 :C\rightarrow S\}$
where $\pi_1:C\rightarrow S$ is a flat family of conics contained in $\PP^n_S$ and passing through $p$ and $q$.
\end{defi}
In other words, the moduli functor $M$ parametrizes conics contained in $\PP^n$ and passing through $p$ and $q$.

\begin{lemm}\label{lemM}
The functor $M$ is representable by a scheme. We denote it by $M$ as well. The scheme $M$ is a $\PP^3$-bundle over $\PP^{n-2}$.
\end{lemm}

\begin{proof}
%Let $V$ be the open subscheme of $\Hilb^{2t+1}(\PP^n)$ which parametrizes reduced conics contained in $\PP^n$. Suppose that $g:C\rightarrow V$ is the universal family of reduced conics contained in $\PP^n$. It is clear that $C\subseteq V \times \PP^n $. Let $\pr_2$ be the second projection of $V\times \PP^n $, and $I_p$ (resp. $I_q$) be the scheme $C\cap \pr_2^{-1}\{p\}$ (resp. $C\cap \pr_2^{-1}\{q\}$) .

%The scheme theoretical intersection $I_p\cap C_s$ is empty or a closed point of the curve $C_s$ for a closed point $s\in V$. Therefore, the restriction $g|_{I_p}$ is proper, one to one, and unramified. In particular, the morphism $g|_{I_p}$ is a closed immersion. It is clear that the closed subscheme $g(I_p)$ parametrizes reduced conics contained in $\PP^n$ and passing through $p$. Similarly, it is the same for the scheme $g(I_q)$. The scheme theoretical intersection of $g(I_p)$ and $g(I_q)$ gives the scheme $M^{rep}$ representing the functor $M$.

The lemma follows from the following observation. It is well known that a conic $C$ in $\PP^n$ is contained in a unique projective 2-plane $\Span(C)=\PP^2$ as a subscheme. On the other hand, it is clear that $\PP^{n}/\overline{pq}=\PP^{n-2}$ parametrizes 2-planes containing $\overline{pq}$. Note that conics contained in a 2-plane $\PP^2$ and containing $p$ and $q$ are parametrized by $\PP^3 $. It follows that $M$ is a $\PP^3$-bundle over $\PP^{n-2}$.

\end{proof}
We say that a conic in $\PP^n$ is good if it is smooth or the line $\overline{pq}$ is not one of its components. Otherwise, we say that the conic is bad. In particular, we have that $\overline{pq}\subseteq C$ for a bad conic $C$. Note that the space of conics contained in a 2-plane and passing through $p$ and $q$ is $\PP^3$ and the locus of bad conics in this 2-plane is $\PP^2(\subseteq \PP^3)$. It follows that the locus $B$ of bad conics in $M$ is a divisor representing the relative $\OO(1)$ of the $\PP^3$-bundle $M$ over $\PP^{n-2}$. It implies the following lemma.
\begin{lemm} \label{LemmATB}
With the notation as above, there exists an open subscheme $M^{o}$ of $M$ parameterizing reduced good conics contained in $\PP^n$. Moreover, we have that $M^{o}=M-B$ is an $\mathbb{A}^3$-bundle over $\PP^{n-2}$.

\end{lemm}

\begin{lemm}\label{ktimes}
Let $\Delta_1$ be the boundary divisor of $M^o$ parametrizing good reducible conics contained in $\PP^n$ and passing through $p$ and $q$. The \text{Picard} group $\Pic(M^o)$ of $M^o$ is $\mathbb{Z}$. In particular, we have that\[[\Delta_1]=m\epsilon \text{ and }\psi^*( \mathcal{O}_{\mathbb{P}^{n-2}}(1))=h\epsilon\] for some integers $m$ and $h$, where $\epsilon$ is a generator of $\Pic(M^{o})$.

\end{lemm}
\begin{proof}
The lemma follows from Lemma \ref{LemmATB}.
\end{proof}
%\begin{proof}
%By Lemma \ref{lemlb}, the scheme $M^o$ is a line bundle over $\Delta_1$. By the smoothness of $M^o$ and $\Delta_1$ (Lemma \ref{lemMsm} and Lemma \ref{lemdel1}), the Picard group $\Pic(M^o)$ is equal to $\Pic(\Delta_1)$ via the pullback $r^*:\Pic(\Delta_1)\rightarrow \Pic(M^o)$. Since $\Delta_1$ is $\mathbb{P}^n-\overline{pq}$ and $n\geq 3$, the Picard group $\Pic(\Delta_1)$ is equal to $\Pic(\PP^n)=\mathbb{Z}$. We prove the lemma.

%\end{proof}
\section{The Smoothness and Chern Classes} \label{sectionofsm}
In this section, we prove the smoothness of the universal bundle $\U$ by local calculations. We also prove some relations between the Chern classes of $\U$, $\F$ and the algebraic cycle $[\Z]$ associated to the singular locus $\Z$ of the map $\pi:\U\rightarrow \F$. The results in this section are preparations for applying the Grothendieck-Riemann-Roch theorem in the next section. We suggest the reader to skip this section on a first reading and return back if necessary.

\begin{defi}
With the same notation as before, we consider the universal family $$\pi:\U\rightarrow \F$$of conics. The singular locus $\Z$ of $\pi$ is defined by the first fitting ideal of the relative differential sheaf $\Omega_{\U/\F}^1$.
\end{defi}
It is clear that $\Z$ is the locus of the nodal points of the family $\U\rightarrow \F$ of conics.
\begin{lemm} \label{dif} \label{lemm11} \label{localnodal}
Suppose that $z\in \U$ is the nodal point of the reducible conic $\U_{\pi(z)}$. Let $\widehat{\OO_{\U,z}}$ be the completion of the local ring $\OO_{\U,z}$ at $z$. We have that\[\widehat{\OO_{\U,z}}=\widehat{\OO_{\F,\pi(z)}}[[x,y]]/(xy-a)\] where $a$ is an element in $\widehat{\OO_{\F,\pi(z)}}$. The singular locus $\Z$ is defined by the ideal $(x,y)$ in $\widehat{\OO_{\U,z}}$. Moreover, we have that \[\Omega_{\U/\F}^1=\omega_{\U/\F}\otimes I_{\Z}.\]where $\omega_{\U/\F}$ is the dualizing sheaf of $\pi$ and $I_{\Z}$ is the ideal sheaf of $\Z$
\end{lemm}

\begin{proof}
The first statement is well known. The analytic version follows from Proposition 2.1 in \cite[Chapter X]{GH} . Since the general fibers of $\pi$ are smooth conics contained in $X$, the second statement follows from the identity in the paper \cite[Page 101 ii)]{D}.
\end{proof}

\begin{prop}\label{sm2}
The singular locus $\Z$ is a smooth subvariety of codimension 2 in $\U$. Moreover, it is isomorphic to the boundary divisor $\Delta$ via the morphism $\pi$.
\end{prop}

\begin{proof}
Let $z$ be a point of $\Z$. By Lemma \ref{lemm11}, the total family $\U$ is defined by the equation \[xy-a=0\] in a neighborhood of the point $z=(0,0, s)\in \mathbb{C}^2\times \F $, where $a$ is an analytic function on $\F$ and $x,y$ are coordinates of $\mathbb{C}^2$. The singular locus $\Z$ in $\U$ is defined by $$x=0 \text{~and~} y=0.$$ Therefore, the locus $\Z$ is defined by \[xy-a=0, x=0\text{~and~}y=0\]in a neighborhood of $(0,0, s)\in \mathbb{C}^2\times \F$. We conclude that the locus $\Z$ is isomorphic (via the morphism $\pi$) to the analytic subspace of $\F$ defined by $a=0$. On the other hand, the boundary divisor $\Delta$ in $\F$ is defined by $a=0$ as well. We have proved the lemma.
\end{proof}

\begin{prop}\label{sm1}
The universal family $\U$ is a smooth projective variety.
\end{prop}

\begin{proof}

Since $\F$ is smooth, it suffices to show that $\U$ is smooth at $\Z$. Let $z$ be a point of $\Z$. By Lemma \ref{lemm11} and the proof of Proposition \ref{sm2}, the total family $\U$ is defined by \[xy-a=0\text{~(}a\in \OO_{\F}\text{)}\]in a neighborhood of the point $z=(0,0,s)\in \mathbb{C}^2\times \F$ and the boundary divisor $\Delta$ is locally defined by $a=0$ in $\F$.

To prove the proposition, it suffices to show that the Jacobian
\[(y,x,-\frac{\partial a}{\partial z_1},\ldots, -\frac{\partial a}{\partial z_n})\]
of the funtion $xy-a$ does not vanish at any point of $\U$, where $\{z_i\}$ are the local coordinates of $\F$. Since the smooth divisor $\Delta$ in $\F$ is defined by $a=0$, the Jacobian \[(\frac{\partial a}{\partial z_1},\ldots, \frac{\partial a}{\partial z_n})\]of $a$ on $\F$ does not vanish anywhere. We have proved the proposition.

\end{proof}

\begin{lemm}\label{td}We have that
\[td_1(\Omega^1_{\U/\F})=\frac{1}{2} c_1(\Omega^1_{\U/\F})\text{~and~} td_2(\Omega^1_{\U/\F})=\frac{1}{12}(c_1^2(\Omega^1_{\U/\F})+c_2(\Omega^1_{\U/\F}))\]
where $\Omega^1_{\U/\F}$ is the relative differential sheaf of $\pi$ and $td_i$ is the i-th Todd Class.

\end{lemm}
\begin{proof}
See \cite[Appendix A]{H} or \cite{F}.
\end{proof}

\begin{lemm}\label{td2}
Let $T_{\pi}$ be the relative tangent sheaf of $\pi$. The i-th Todd Class $td_i(T_{\pi})$ is equal to $(-1)^i td_i(\Omega_{\U/\F}^1)$ where $td(T_{\pi})$ is given by $\frac{td(T_{\U})}{\pi^* td(T_{\F})}$.
\end{lemm}

\begin{proof}
We denote $td_i(T_{\pi})$ by $td_i$. It is clear that
\[ (1+td_1+td_2+\dots)\pi^*td(T_{\F})=td(T_{\U}).\]
From the short exact sequence $0\rightarrow \pi^*\Omega^1_{\F}\rightarrow \Omega^1_{\U} \rightarrow \Omega^1_{\U/\F}\rightarrow 0$, we have that \[td(\Omega^1_{\U})=\pi^*td(\Omega^1_{\F})  td(\Omega^1_{\U/\F}).\]
On the other hand, we have that $td_i(T_{\F})=(-1)^i td_i(\Omega^1_{\F})$ and $td_i(T_{\U})=(-1)^i td_i(\Omega^1_{\U})$. We conclude that $td_i(T_{\pi})=(-1)^i td_i(\Omega^1_{\U/\F})$.
\end{proof}
\begin{lemm}\label{lemchom}  \label{lemchiz}
Let $\omega_{\U/\F}$ be the dualizing sheaf of the morphism $\pi$. We have the following identities

\[ch_0(\Omega^1_{\U/\F})=1,c_1(\Omega^1_{\U/\F})=c_1(\omega_{\U/\F}),\] \[ch_1(\Omega^1_{\U/\F})=c_1(\omega_{\U/\F})+c_1(I_{\Z})=c_1(\omega_{\U/\F}).\]
Let $I_{\Z}$ be the ideal sheaf of the singular locus $\Z$ in $\U$. We have that
\begin{center}
$ch_0(I_{\Z})=1$, $c_1(I_{\Z})=0$, $ch_1(I_{\Z})=0$,\\ $c_2(I_{\Z})=[\Z]$, $ch_2(I_{\Z})=-c_2(I_{\Z})=-[\Z]$,\\
$ch_2(\Omega^1_{\U/\F})=\frac{1}{2}(c_1(\omega_{\U/\F}))^2-[\Z]$ and $c_2(\Omega^1_{\U/\F})=[\Z]$,
\end{center}
where $[\Z]$ is the fundamental class of $\Z$.

\end{lemm}

\begin{proof}
Let $i$ be the natural inclusion $i: \Z\subseteq \U$. Using Lemma \ref{dif} and the fact that $\omega_{\U/\F}$ is a line bundle on $\U$, we have \[ch(\Omega^1_{\U/\F})=ch(\omega_{\U/\F}\otimes I_{\Z})=ch(\omega_{\U/\F})ch(I_{\Z}).\]
Expanding the right side, we show the first assertion. Since we know that $\U$ and $\Z$ are smooth from Proposition \ref{sm1} and Proposition \ref{sm2}, we can apply the formula in \cite[Example 15.3.5 Page 298]{F}. Therefore, we have that
$$\begin{array}{lll}
c(i_*(\mathcal{O_{\Z}})) & = & 1-i_*(c(N_{\Z/\U}^{\vee})^{-1}) \\
& = &  1-i_*(1-c_1(N_{\Z/\U}^{\vee}))+\dots \\
&=& 1-[\Z]-i_*(c_1(N_{\Z/\U}))+\dots,
\end{array}$$
where the notation $"\dots"$ means the cycles of higher codimension.
From the exact sequence
\[0\rightarrow I_{\Z} \rightarrow \mathcal{O}_{\U}\rightarrow i_* \mathcal{O}_{\Z}\rightarrow 0 ,\] we get the identity \[1=c(I_{\Z})c(i_*\mathcal{O}_{\Z})\] by the Whitney product formula. Expanding the right hand side and comparing the terms with the left hand side, we have proved the lemma .
\end{proof}

\section{An Integral Cycle Relation}
In this section, we apply the Grothendieck-Riemann-Roch theorem to show an identity relating divisors on $\F$, cf. Lemma \ref{twotimes}. Then we use the results in Section \ref{sectionofconics} to show an integral version of Lemma \ref{twotimes}, cf. Proposition \ref{asdiv}.

We denote by $\lambda$ the first Chern Class of line bundle $\varphi^*(\mathcal{O} _{\mathbb{P}^{n-2}}(1))$. Recall that we have the universal morphism\[f:\U\rightarrow X\] from the universal conic $\U$ to $X$, cf. Section 2. We hope it will cause no confusion if we sometimes also denote by $f:\U\rightarrow \PP^n$ the morphism $\U\rightarrow X\subseteq \PP^n$.

\begin{lemm}\label{crlemm1}

We have the following identities
\begin{enumerate}

\item $\pi^*\varphi^*(\mathcal{O} _{\mathbb{P}^{n-2}}(1))=\omega_{\U/\F}\otimes f^*\mathcal{O}_{\mathbb{P}^n}(1)$,
\item $c_1(\sigma_i^{*}T_{\U/\F})=-\lambda$,
\item $c_1(\omega_{\U/\F})=\pi^*\lambda - c_1(f^*\mathcal{O}_{\mathbb{P}^n}(1))$.

\end{enumerate}
\end{lemm}
These identities are scattered in the paper \cite{DS}. We give a geometric proof.
\begin{proof}
Since the fiber $\U_x$ of $\pi$ over $x\in \F$ is a conic, the dualizing sheaf $\omega _{\U_x}$ is $\OO_{\U_x}(-2)=(f|_{\U_x})^*(\mathcal{O}_X(-1))$. For the family $\pi:\U\rightarrow \F$ of conics, the line bundle $\omega_{\U/\F} \otimes f^*\mathcal{O}_X(1)$ is fiberwise trivial. By the base change theorem, there is a line bundle $L$ on $\F$ such that
\begin{equation} \label{sh}
{\pi^*L=\omega_{\U/\F} \otimes f^*\mathcal{O}_X(1)}.
\end{equation}
Since the image of $f\circ \sigma_0$ is the point $p$, the pullback $(f\circ \sigma_0)^*\mathcal{O}_{X}(1)$ is trivial. Therefore, we have that \[L=\sigma_0^*(\pi^*L)=\sigma_0^*\omega_{\U/\F}\otimes (f\circ \sigma_0)^*\mathcal{O}_X(1)=\sigma_0^*\omega_{\U/\F}.\]
Since the morphism $\pi$ is smooth at the image of $\sigma_0$, we have that\[\sigma^*_0 \omega^{\vee}_{\U/\F}=\sigma_0^*\Omega_{\U/\F}^{\vee}=\sigma_0^*T_{\U/\F}.\]

For the conic $\U_x$, the 2-plane $\Span(\overline{pq},T_p \U_x)$ coincides with the 2-plane  $\Span(\U_x)$, where $T_p\U_x$ is the tangent line to $\U_x$ at the point $p$. On the other hand, we are able to construct a map $\rho:\F\rightarrow \mathbb{P}^{n-2}$ which associates to a point $$x=[\U_x,\U_x\subseteq X,p,q]\in \F$$ the point $[\Span(p,q,T_p\U_{x})/\Span(p,q) ]\in\PP^{n-2}=\PP^n/\Span(p,q)$, cf. (\ref{vmap}). So we have that\[\sigma_0^*T_{\U/\F}=\rho^*\mathcal{O}_{\mathbb{P}^{n-2}}(-1).\] It is easy to see that the morphism $\rho$ is exactly the morphism $\varphi$. We show $(2)$. Moreover, we have that
\[L=\sigma_0^*\omega_{\U/\F}=\rho^*\mathcal{O}_{\mathbb{P}^{n-2}}(1)=\varphi^*\mathcal{O}_{\mathbb{P}^{n-2}}(1).\]
Combining with equality \eqref{sh}, we show $(1)$ and $(3)$.

\end{proof}

We denote by $\sigma_i$ the image $\sigma_i(\F)$ of the section $\sigma_i$ as well.
\begin{lemm} \cite[Lemma 6.4]{DS}\label{detl} \label{crlemtc}  \label{crlemds}
We have that
\begin{enumerate}

\item $[\sigma_i]^2=-\sigma_{i*}(\lambda)$.
\item $\pi_*f^*(c_1(\mathcal{O}_{\mathbb{P}^n}(1))^2)=2\lambda$ in $\mathrm{CH}^1(\F)_{\mathbb{Q}}$.
\item $c_1(f^*\OO_{X}(1))=[\sigma_0]+[\sigma_1]+\pi^*\lambda$ (in page 33 of \cite{DS}).
\end{enumerate}
\end{lemm}
\begin{proof}
We sketch a proof here for the sake of completeness. From the proof of Lemma \ref{crlemm1}, we have that \[\sigma^*_i(T_{\U/\F})=\rho^*\OO_{\PP^{n-2}}(-1)=-\varphi^*(\OO_{\PP^{n-2}}(1))=-\lambda.\]Since $\sigma_i$ is a smooth divisor of $\U$, we have that
\[ [\sigma_i]^2= \sigma_{i*}(c_1(N_{\sigma_i/\U}))=\sigma_{i*}(\sigma_i^*(T_{\U/\F}))=-\sigma_{i*}(\lambda).\]
One can show (3) if one knows the precise definition of $\varphi$. We refer to \cite[Lemma 6.4 page 33]{DS}. Here, we provide an alternative way to show (3) from Lemma \ref{var}. Let $H$ be a hyperplane in $\mathbb{P}^{n-2}$. We consider the rational pullback $(\pr|_X)^*(H)$ of $H$ via the map $\pr|_X$ in Lemma \ref{var}. The pullback $(\pr|_X)^*(H)$ is a hyperplane contained in $\PP^n$ and passing through $p$ and $q$.
Therefore, it follows that \[f^*c_1(\OO_X(1))=f^*((\pr|_X)^*H)=\pi^*\varphi^*\OO_{\PP^{n-2}(1)}=\pi^*\lambda\]holds on $\U-\sigma_0-\sigma_1$ from the commutativity of the diagram in Lemma \ref{var}. We conclude that \[f^*c_1(\OO_X(1))=\pi^*\lambda+A[\sigma_0]+B[\sigma_1]\] for some positive integers $A$ and $B$ since $(\pr|_X)^*(H)$ passes through $p$ and $q$. Let $C$ be a fiber of $\pi$. The intersection number of $C$ and $f^*c_1(\OO_X(1))$ is two since $C$ is a conic in $X$. Meanwhile, we have that\[C\cdot \pi^*(\lambda)=0 \text{~and~} C\cdot [\sigma_i]=1.\] We conclude that $A=B=1$. We have proved (3).

The second assertion follows from (1) and (3).
\end{proof}

\begin{lemm}\label{cycle}
With the notation as before, we have that
\begin{equation} \label{equaS}
[c_1(\omega_{\U/\F})]^2=-\pi^*(\lambda^2) +f^*(c_1(\mathcal{O}_{\PP^n}(1))^2)-2\pi^*(\lambda)([\sigma_0]+[\sigma_1]).
\end{equation} In particular, we have that $\pi_{*}([c_1(\omega_{\U/\F})]^2)=-2\lambda$ in $\mathrm{CH}^1(\F)_{\mathbb{Q}}$.
\end{lemm}

\begin{proof}
By Lemma \ref{crlemm1} (3), we have that\[[c_1(\omega_{\U/\F})]^2=\pi^*(\lambda^2)+f^*(c_1(\mathcal{O}_{\PP^n}(1))^2)-2\pi^*(\lambda)\cdot f^*(c_1(\mathcal{O}_{\PP^n}(1)).\] By Lemma \ref{crlemds}, we know that \[c_1(f^*\mathcal{O}_{\mathbb{P}^n}(1))=[\sigma_0]+[\sigma_1]+\pi^*\lambda.\]
Hence, we show the first assertion by combining these two identities. Applying $\pi_*(-)$ to the identity (\ref{equaS}), we show the second assertion by the projection formula and the fact $\pi_*([\sigma_i])=[\F]$.
\end{proof}

\begin{lemm}\label{twotimes}
We have that $\Delta=2\lambda$ in $\Pic^1(\F)_\mathbb{Q}$.
\end{lemm}
\begin{proof}
By Proposition \ref{sm1}, we can apply the Grothendieck-Riemann-Roch theorem \cite[Chapter 15]{F} to the morphism $\pi:\U\rightarrow \F$. Denote by $\omega$ the dualizing sheaf $\omega_{\U/\F}$. By Lemma \ref{td}, Lemma \ref{td2} and Lemma \ref{lemchom}, we conclude that $(td(T_{\pi}))_{\leq 2}$ (up to degree 2) is equal to
$$\begin{array}{ll}
  1-td_1(\Omega^1_{\U/\F})+td_2(\Omega^1_{\U/\F}) & =1-\frac{1}{2}c_1(\Omega^1_{\U/\F})+\frac{1}{12}(c_1(\Omega^1_{\U/\F})^2+c_2(\Omega^1_{\U/\F})) \\
  \\
   & =1-\frac{1}{2}c_1(\omega)+\frac{1}{12}(c_1(\omega)^2+[\Z]).
\end{array}$$

By Lemma \ref{lemchiz}, we have that $ch(I_{\Z})_{\leq 2}=1-[\Z]$ (up to degree 2).
Therefore, the term of degree 2 in $td(T_{\pi})\cdot ch(I_{\Z})$ is
\begin{equation} \label{degree2}
(td(T_{\pi})\cdot ch(I_{\Z}))_2=\frac{1}{12}[\Z]+\frac{1}{12}c_1^2(\omega)-[\Z].
\end{equation}

Let $\pi_!(G)$ be $\sum \limits_{i=0}^{dim(\U)}(-)^iR^i\pi_*(G)$ (in the K-group $K_0(\F)$ of $\F$) for a coherent sheaf $G$ on $\U$. We claim that \begin{enumerate}
                                 \item $\pi_!\mathcal{O}_{\U}=\mathcal{O}_{\F}$,
                                 \item and $\pi_!(i_*\mathcal{O}_{\Z})=j_*(\mathcal{O}_{\Delta})$ where the maps $i:\Z\rightarrow \U$ and $j:\Delta \rightarrow \F$ are the closed immersions.
                               \end{enumerate}
In fact, suppose that $C$ is a reducible  conic with a node point $b$, we have a short exact sequence
\[0\rightarrow \mathcal{O}_{P_1}(-b)\rightarrow \mathcal{O}_{C} \rightarrow \mathcal{O}_{P_2} \rightarrow 0,\]
where $P_1$ and $P_2$ are the distinct components of $C$. From the short exact sequence, we have that $H^1(C,\mathcal{O}_C)=0$. On the other hand, it is obvious that $H^1(C,\mathcal{O}_C)=0$ if $C$ is a smooth conic. By the base change theorem, we have that $R^i \pi_*\mathcal{O}_{\U}=0$ for all $i\geq 1$. We conclude that $\pi_!\mathcal{O}_{\U}=\mathcal{O}_{\F}$.

To prove $\pi_!\mathcal{O}_{\Z}=\mathcal{O}_{\Delta}$, it suffices to show that $R^k \pi_*\mathcal{O}_{\Z}=0$ for all $k\geq 1$. In fact, by Proposition $\ref{sm2}$, the map $\pi\circ i:\Z \rightarrow \Delta$ is an isomorphism. It is clear that \[R^{s}j_{*}(\OO_{\Delta})=0 \text{~(resp.~}R^{l}i_*(\OO_{\Z})=0\text{)}\] if $s\geq 1$ (resp. $l\geq 1$). In particular, the Leray spectral sequence  \[E^{k,l}_2=R^k \pi_*(R^l i_*(\mathcal{O}_{\Z}))\Rightarrow R^{k+l} (j\circ \delta)_*(\mathcal{O}_{\Z})=R^{k+l} j_*(\mathcal{O}_{\Delta})\]
degenerates at the $E_2$ page. We conclude that $R^k \pi_*(i_*\mathcal{O}_{\Z})=0$ for $k\geq 1$.
Applying the Grothendieck-Riemann-Roch theorem, we have that
\begin{equation} \label{eqgroth}
\pi_*(ch(I_{\Z})\cdot td(T_{\pi}))=ch(\pi_!(I_{\Z}))=ch(\pi_!(\mathcal{O}_{\U})-\pi_!(i_*(\mathcal{O}_{\Z}))).
\end{equation}
The right hand side RHS of (\ref{eqgroth}) is equal to
\[ch(\mathcal{O}_{\F}-j_*\mathcal{O}_{\Delta})=ch(\mathcal{O}_{\F}(-\Delta)).\]
In particular, the term of degree $1$ in RHS is equal to $-[\Delta]$. By Proposition \ref{sm2} and the equality \eqref{degree2}, the term $\pi_*((td(T_{\pi})\cdot ch(I_{\Z}))_2)$ of degree $1$ in the left hand side (LHS) of (\ref{eqgroth}) is equal to \[\frac{1}{12}[\Delta]+\frac{1}{12}\pi_*(c_1^2(\omega))-[\Delta] .\]Comparing the term of dgree 1 in RHS with the term in LHS, we conclude that \[0=\frac{1}{12}[\Delta]+\frac{1}{12}\pi_*(c_1^2(\omega))\]in $\mathrm{CH}^1(\F)_{\mathbb{Q}}$. Therefore, the divisor $\Delta$ is equal to $2\lambda$ in $\Pic^1(\F)_{\mathbb{Q}}$ by Lemma \ref{cycle}.

\end{proof}
In the following, we use the notation in Section \ref{sectionofconics} freely. For the smooth scheme $M^o$ (cf. Lemma \ref{LemmATB}), we have a morphism
\begin{equation} \label{ph}
\psi: M^o\rightarrow \mathbb{P}^{n-2}=\PP^n/\Span(p,q)
\end{equation}
that associates to a point $[C]\in M^o$ the point \[\psi([C])=[\Span(C)] \in \PP^{n-2}=\PP^n/\Span(p,q),\] cf. (\ref{vmap}).

\begin{propo} \label{asdiv}
The boundary divisor $\Delta$ of $\F$ is linearly equivalent to $2\lambda$.
\end{propo}

\begin{proof}
From the moduli interpretation, we have a morphism $H:\F\rightarrow M^o$ with $\varphi=\psi \circ H.$
\[\xymatrix{\F \ar[r]^{H} \ar[rd]^{\varphi} & M^o\ar[d]^{\psi} \\
&\PP^{n-2}}\]With the notation in Lemma \ref{ktimes}, it is clear that the divisor $\Delta$ is the pullback of $\Delta_1$ (Lemma \ref{ktimes}) via $H$. Let $k$ be the rational number $\frac{m}{h}$. By Lemma \ref{ktimes} and the diagram above, we have that \[\Delta=H^*(\Delta_1)=k\cdot H^*(\psi^*(\mathcal{O}_{\PP^{n-2}}(1)))=k \varphi^*(\mathcal{O}_{\PP^{n-2}}(1))=k \lambda\]in $\Pic(\F)_{\mathbb{Q}}$.
By Lemma \ref{twotimes}, we have that $\Delta=2\lambda$ in $\Pic(\F)_{\mathbb{Q}}$. Therefore, we conclude that $k=2$.
In particular, we have that $\Delta_1=2 \psi^*(\mathcal{O}_{\PP^{n-2}}(1))$ in $\Pic(M^o)$. Therefore, we conlcude that \[\Delta=H^*(\Delta_1)=2 H^*(\psi^*(\mathcal{O}_{\PP^{n-2}}(1)))=2 \varphi^*(\mathcal{O}_{\PP^{n-2}}(1))=2 \lambda\] in $\Pic(\F)$. We have proved the proposition.
\end{proof}
\begin{remark}
Proposition \ref{asdiv} could be proved alternatively by a test curve computation (see \cite[Lemma 4.12]{ZC}).
\end{remark}

\section{The Main Theorem}
We start with some lemmas to show a criterion (Proposition \ref{mainprop1}) to characterize when a projective variety is a complete intersection in a projective space.
\begin{lemm} \label{alg}
Let $R$ be a Cohen-Macaulay ring. If $R$ has only one minimal prime ideal p and the localization $R_p$ is reduced, then R is reduced.
\end{lemm}

\begin{proof}
We consider the localization map $f:R\rightarrow R_p$. We claim that f is injective.

In fact, the ring $R$ has no embedded associated primes by the unmixedness theorem \cite[Theorem 17.6]{Ma}. In particular, from the primary decomposition of the zero ideal, we know that the zero ideal is a p-primary ideal. Since the kernel $\ke(f)$ of $f$  is the smallest $p$-primary ideal by \cite[Exercise 4.11]{A}, the kernel $\ke(f)$ is a subset of the zero ideal, i.e., the map $f$ is injective. We have proved the claim.

Therefore, the ring $R$ is reduced since $R_p$ is reduced.
\end{proof}

\begin{lemm} \label{char}
 Suppose that $X$ and $Y$ are smooth projective varieties in $\PP^n$. Assume that
 \begin{enumerate}
   \item $X$ is a divisor of $Y$,
   \item and $X$ is a complete intersection in $\PP^n$ defined by homogeneous polynomials $(f_1,f_2,\dots,f_m)$,
   \item and the line bundle $\mathcal{O}_Y(X)$ is equal to $\mathcal{O}_Y(\degr(f_m))$ in $\Pic(Y)_{\mathbb{Q}}$,
   \item and $f_1|_Y=f_2|_Y=\dots=f_{m-1}|_Y=0$.
 \end{enumerate}
 Then the variety Y is a complete intersection contained in $\PP^n$ and defined by homogeneous polynomials $(f_1,f_2,\dots,f_{m-1})$.
\end{lemm}

\begin{proof}
Let $Z$ be a subscheme of $\PP^n$ defined by the polynomials $(f_1,f_2,\dots,f_{m-1})$. Denote by $V(f_1,f_2,\dots,f_{m-1})$ the scheme $Z$. By the assumption (2), the scheme $Z$ is equidimensional and of codimension $m-1$. Moreover, the reduced scheme $Z_{red}$ is connected. Since \[X=V(f_1,f_2,\dots,f_m)\] is a divisor of Y and \[f_1|_Y=f_2|_Y=\dots=f_{m-1}|_Y=0,\] the variety Y is one of the irreducible components of the reduced scheme $Z_{red}$ associated to $Z$.

We claim that $Z_{red}$ has only one irreducible component. In particular, we have $Z_{red}=Y$.

In fact, suppose that $Z_{red}=Y \cup W$ where $W$ is the union of other components rather than $Y$. The intersection $Y\cap W$ is not empty since $Z_{red}$ is connected. Moreover, the scheme $Z$ is not smooth at any point in $Y\cap W$. So the hypersurface $V(f_m)$ defined by $f_m=0$ does not meet $Y\cap W$, otherwise, the variety $X=V(f_m) \cap Z$ is not smooth at any point of $V(f_m)\cap Y\cap W$ by local calculations. We conclude that
$$X=V(f_m)\cap Z=(V(f_m)\cap Y) \cup (V(f_m)\cap W)$$ is a non-trivial decomposition, i.e., the variety $X$ is reducible. It is a contradiction.

We claim that the subscheme $Z$ of $\PP^n$ is $Y$ $(\subseteq \PP^n)$.

In fact, we have already proved that $Y=(Z)_{red}$. In particular, the reduced scheme $Z_{red}$ is irreducible. We only need to show that $Z$ is reduced. Since $Z$ is a complete intersection in a projective space, it is Cohen-Macaulay. By Lemma \ref{alg}, if $Z$ is reduced at the generic point, then $Z$ is reduced. To prove that $Z$ is reduced at the generic point, it suffices to prove that $[Z]=[Z_{red}]$ where $[Z]$ is the fundamental class of $Z$ in $\mathrm{CH}^*(\PP^n)$, cf.  \cite[Chapter 1]{F}. We have that $[Z]=k [Z_{red}]$ for some $k\in \mathbb{N}$ and $$[Z]\cdot \mathcal{O}_Y(\degr(f_m))=[X]=[Y]\cdot\mathcal{O}_Y(X).$$Therefore, we get $k [Z_{red}]\cdot\mathcal{O}_Y(\degr(f_m))=[Z_{red}]\cdot\mathcal{O}_Y(X)$.
By the assumption (3) of the lemma, we imply that k=1. We have proved the lemma.
\end{proof}

\begin{prop} \label{mainprop1}
Suppose that $\Delta$ and $\F$ are smooth projective varieties in $ \mathbb{P}^N$. Assume that
\begin{enumerate}
  \item the variety $\Delta$ is a divisor of $\F$ and the dimension of $\Delta$ is at least one;
  \item the divisor $\Delta$ is a complete intersection in $\mathbb{P}^N$ of type $(d_1,\ldots,d_c)$ where $d_i\geq1$;
   \item the divisor $\Delta$ is defined by a global section $\overline{Q}\in H^0(\F,\OO_{\F}(d_1))$ such that $\overline{Q}=Q|_{\F}$ for a global section $Q\in H^0(\PP^N,\OO_{\PP^N}(d_1))$.
\end{enumerate}
Then the fiber $\F$ is a complete intersection in $\mathbb{P}^N$ of type $(d_2,\ldots, d_c)$.
\end{prop}

\begin{proof}
By the assumption (2), we can assume that the divisor $\Delta$ is defined by the polynomials $(F_1,\ldots,F_c)$ where $F_i(X_0,\ldots,X_N)$ is a homogeneous polynomial of degree $d_i$. By the assumption (3), we have a short exact sequence
\[
\xymatrix {
  0\ar[r]& \mathcal{O}_{\F}(-d_1) \ar[r]^{\_\cdot \overline{Q}}& \mathcal{O}_{\F}\ar[r]&\mathcal{O}_{\Delta} \ar[r]&0}.
\]

%By Lemma \ref{surjective} below, we can choose $Q\in H^0(\mathbb{P}^N,\mathcal{O}_{\mathbb{P}^N}(d_1))$ such that $Q|_\F=\overline{Q}$.
Since $\overline{Q}$ defines $\Delta$, we have $Q\in H^0(\mathbb{P}^N, I_{\Delta}(d_1))$ where $I_{\Delta}$ is the ideal sheaf of $\Delta$ in $\PP^N$.

I claim that $F_i|_{\F}=0$ if the degree $d_i$ of $F_i$ is less than $d_1$. In fact, we have the following diagram
\begin{equation} \label{eq:diag}
\xymatrix{0 \ar[r]& \mathcal{O}_{\F}(-d_1) \ar@/^/@{.>}[rdrdr]^{\simeq} \ar[r]^{\underline{\ } \cdot \overline{Q}}& \mathcal{O}_{\F} \ar[r]& \mathcal{O}_{\Delta} \ar[r]& 0\\
&0 \ar[r] & \mathcal{O}_{\mathbb{P}^n} \ar@{->>}[u] \ar@{=}[r] & \mathcal{O}_{\mathbb{P}^n}\ar@{->>}[u] \ar[r] &0\\
&0\ar[r] & I_{\F} \ar@{^{(}->}[u] \ar[r]^i & I_{\Delta}  \ar@{^{(}->}[u] \ar[r]^{\delta}& coker(i)\ar[r]& 0. \\
}
\end{equation}

Since the dotted connecting map in the diagram is an isomorphism, we have a short exact sequence
\[0\rightarrow I_{\F}\rightarrow I_{\Delta} \rightarrow \mathcal{O}_{\F}(-d_1)\rightarrow 0.\]
For an integer $m$, the short exact sequence induces the exact sequence of global sections
\begin{equation}    \label{eq:exact}
\xymatrix{0\ar[r] &\Gamma(\mathbb{P}^{N},I_{\F}(m)) \ar[r] &\Gamma(\mathbb{P}^{N},I_{\Delta}(m))\ar[r]^{\delta'} &\Gamma(\F,\mathcal{O}_{\F}(m-d_1))}.
\end{equation}
We take $m=d_i=\degr (F_i)$ in \eqref{eq:exact}. Since $\Gamma(\F,\mathcal{O}_{\F}(d_i-d_1))=0$ ( $d_i<d_1$), we have $\Gamma(\mathbb{P}^{N},I_{\F}(d_i))=\Gamma(\mathbb{P}^{N},I_{\Delta}(d_i))$. We have proved the claim.

Furthermore, the degree of the homogeneous polynomial $Q$ is $d_1$. Therefore, we could write $Q$ as
\begin{equation} \label{eqfi}
Q=\sum\limits_{k=1}^l H_k F_{i_k} \text{~with~} \degr F_{i_k} \leq d_1.
\end{equation}
By the claim above, we conclude that \[\overline{Q}=\overline {Q-\sum\limits_{\deg(F_{i_k})<d_1} H_k F_{i_k}} \in H^0(\F,\OO_{\F}(d_1)).\]
So we can replace $Q$ by \[Q-\sum\limits_{\deg(F_{i_k})<d_1} H_k F_{i_k}\in H^0(\PP^N,I_{\Delta}(d_1)).\]
In other words, we can assume that $F_{i_k}$ in (\ref{eqfi}) is of degree $\degr (F_{i_k})=d_1$, i.e., the polynomials $H_k$ are constants. Since $\overline{Q}\neq0$, we can assume that \[Q=a_1 F_1+\ldots\] where $a_1$ is a nonzero constant. Therefore, the variety $\Delta=V(F_1,F_2,\ldots,F_c)$ can be defined by the polynomials \[(Q, F_2,\ldots, F_c).\]
Suppose that $u$ is the image of $F\in \Gamma(\mathbb{P}^N,I_{\Delta}(m))$ under the map (see (\ref{eq:exact}))
 \[\delta':\Gamma(\mathbb{P}^{N},I_{\Delta}(m))\rightarrow \Gamma(\F,\mathcal{O}_{\F}(m-d_1))\]
(i.e., $u=\delta'(F)\in \Gamma(\F,\mathcal{O}_{\F}(m-d_1))$). If $m\geq d_1$, then we have \[F|_{\F}=u\cdot (Q|_{\F})\]by the diagram chasing of the diagram (\ref{eq:diag}). Moreover, by Lemma \ref{surjective} below, there is $U\in \Gamma(\mathbb{P}^{N},\OO_{\PP^N}(m-d_1))$ such that $U|_{\F}=u$. Therefore, we have that
 \begin{center}
 $F-U \cdot Q \in \Gamma(\mathbb{P}^N, I_{\F}(m))$ , i.e., $(F-U \cdot Q)|_{\F}=0$.
 \end{center}
Suppose that we take $F$ to be $F_i$ whose degree $\geq d_1$. We have homogeneous polynomial $U_i$ such that
\begin{equation} \label{eq:big}
(F_i-U_i\cdot Q)|_{\F}=0.
\end{equation}
Suppose that the degrees of $F_2,\ldots, F_l$ are at least  $d_1$ and the degrees of $F_ {l+1},\ldots, F_c $ are less than $d_1$. We conclude that the divisor $\Delta$ can be defined by homogeneous polynomials  \[\left(Q, F_2-U_2 Q, \ldots, F_l-U_l Q, F_{l+1},\ldots, F_{c}\right).\] Therefore, applying Lemma \ref{char} to $\F$, we show that $\F$ is defined by homogeneous polynomials\[(F_2-U_2 Q, \ldots, F_l-U_l Q, F_{l+1},\ldots, F_{c}).\]
In particular, the fiber $\F$ is a complete intersection of type $(d_2,\ldots,d_c)$ in $\mathbb{P}^N$.

\end{proof}

\begin{lemm} \label{surjective}
With the same assumption as in Proposition \ref{mainprop1}, the restriction map $$\Gamma(\mathbb{P}^{N},\mathcal{O}_{\mathbb{P}^{N}}(m)) \rightarrow \Gamma(\F,\mathcal{O}_{\F}(m))$$ is surjective for every $m\in \mathbb{N}$
\end{lemm}

\begin{proof}
  %and $h^1(\Delta, \mathcal{O}_{\Delta}(n))=0$ for all $n\in\mathbb{N}$.
We prove the lemma by induction on $m$.
In fact, we have the following diagram of the short exact sequences
\[\xymatrix{0\ar[r]& \mathcal{O}_{\mathbb{P}^{N}}(-d_1) \ar[d] \ar[r] &\mathcal{O}_{\mathbb{P}^{N}}\ar[d] \ar[r] &i_*\OO_{\Delta}\ar[d] \ar[r]&0&\\
0\ar[r]& \mathcal{O}_{\F}(-d_1)\ar[r] &\mathcal{O}_{\F}\ar[r] &j_*\mathcal{O}_{\Delta}\ar[r]&0
}
\]
where $j$ and $i$ are the natural inclusions.
Therefore, we have that
\[\xymatrix{0\ar[r]& \Gamma(\PP^N,\mathcal{O}_{\mathbb{P}^{N}}(m-d_1)) \ar[d]^h \ar[r] &\Gamma(\PP^N, \mathcal{O}_{\mathbb{P}^{N}}(m))\ar[d]^s \ar[r]^g &\Gamma(\Delta, \mathcal{O}_{\Delta}(m)) \ar@{=}[d]\\
0\ar[r]& \Gamma(\F,\mathcal{O}_{\F}(m-d_1))\ar[r] & \Gamma(\F,\mathcal{O}_{\F}(m))\ar[r] &\Gamma(\Delta, \mathcal{O}_{\Delta}(m))
}
\]
where the rows are exact. Since $d_1\geq 1$ , we know that the map $h$ is surjective by the induction. Since the boundary divisor $\Delta$ is a complete intersection in $\PP^N$ and of dimension at least one, we conclude that the map
$$g:\Gamma(\mathbb{P}^{N},\mathcal{O}_{\mathbb{P}^{N}}(m))\rightarrow \Gamma(\Delta ,\mathcal{O}_{\Delta}(m))$$
is surjective. Therefore, the map $s$ is surjective by the snake lemma. We have proved the lemma.
\end{proof}

%\begin{thm} \label{mainthm1}
%If $n\geq2\sum\limits_{i=1}^{c} d_i -c+1$, then $\F$ is of the expected dimension \[n+1-2\sum\limits_{i=1}^c d_i+c.\] Via the map $\varphi :\F\rightarrow \mathbb{P}^{n-2}$, the general fiber $\F$ is a smooth complete intersection in $\mathbb{P}^{n-2}$ of type
%\[(1,1,2,2,\dots,d_1-1,d_1-1,d_1; 1,1,2,\dots,d_2-1,d_2-1,d_2;\ldots\]
%\[;1,1,2,2,\dots,d_c-1,d_c-1,d_c)-(1,1,2).\]
%where the subtraction means remove two linear polynomials and one quadratic polynomial from the polynomials of type \[(1,1,2,2,\dots,d_1-1,d_1-1,d_1; 1,1,2,\dots,d_2-1,d_2-1,d_2;\ldots\]
%\[;1,1,2,2,\dots,d_c-1,d_c-1,d_c).\]
%In particular, there are $2c-2$ linear polynomials, $2c-1$ quadric polynomials and some polynomials of higher degree contributing to define $\F$.
%\end{thm}
In the end, we show Theorem \ref{mainthmb} as follows.
\begin{proof}
The first assertion follows from Proposition \ref{sm0}. By Proposition \ref{propemb}, we know that $\F$ is a smooth subvariety of $\PP^n$ via the embedding $\varphi$.

By Lemma \ref{delci}, we can suppose that the complete intersection $\varphi(\Delta)$ in $\PP^{n-2}$ is defined by the homogeneous polynomials
\[(Q,F_1,\dots,F_l,H_1,\dots,H_k)\]
where $\degr(Q)=2$, $\degr(H_i)=1$ and $\degr(F_i)\geq2$.

Furthermore, the inequality $$n\geq2\sum\limits_{i=1}^{c}d_i-c+1$$ ensures that the dimension of the divisor $\Delta$ of $\F$ is at least one. We claim that the boundary divisor $\Delta$ is the intersection of $\F$ and a quadric hypersurface in $\PP^{n-2}$. The theorem follows from Proposition \ref{mainprop1} and Lemma \ref{delci} . We show the claim in the following.

By Proposition \ref{asdiv}, we know that the ideal sheaf of $\Delta$ is $\mathcal{O}_{\F}(-2)$. To prove the claim, it suffices to find one quadratic polynomial on $\PP^{n-2}$ which vanishes on the divisor $\Delta$ but not on the entire fiber $\F$.

%  Recall that the morphism $\rho$ which is defined in the proof of Lemma %\ref{crlemm1}, the map
%\[\rho:\F\rightarrow \mathbb{P}^{n-2}\]associates to a point $(C,f,\sigma_0,\sigma_1)$ a point \[
%[\Span(p,q,df_*(T_{\sigma_0}))] \in\PP^{n-2}=\PP^n/\Span(p,q).\]

%The map $\rho$ is the same as the map $\varphi$.

For simplicity, we assume that $X$ is a smooth hypersurface defined by $G(X_0,\ldots,X_n)$. In an affine coordinate system ($\mathbb{A}^n$) whose origin is $q$, we can write $G$ as \[G=f_1+f_2+\ldots+f_d\]
where $f_i(X_1,\ldots,X_n)$ is a homogeneous polynomial of degree $i$ and the tangent space $T_{X,q}$ is defined by $f_1$.

It is well known that the union $L_q$ of lines contained in $X$ and passing through $q$ is defined by the homogeneous polynomials $$(f_2,f_3,\ldots,f_d)$$ in the projetive tangent hyperplanes $\PPT_{X,q}$ (the closure of the tangent space $T_{X,q}$ in $\PP^n$ ), cf. \cite[Lemma 2.1]{Cubic} and Lemma \ref{cilines}.  Moreover, we have $\Delta=L_p\cap L_q$ where $L_p$ is the union of lines contained in $X$ and passing through $p$, cf. the proof of Proposition \ref{sm0}. Therefore, the quadratic polynomial $f_2$ is vanishing on the boundary $\Delta(\subseteq \PPT_{X,q})$.

Recall that the projective space $\PP^{n-2}=\PP^n/\Span(p,q)$ can be identified with the intersection $\PPT_{X,p}\cap\PPT_{X,q}$, cf. Lemma \ref{var}. As above, via the embedding $\varphi$, we can consider the fiber $\F$ to be a subvariety of $\PPT_{X,q}$. In the following, we show that $f_2$ does not vanish on the entire fiber $\F$.

In fact, the polynomial $f_2$ defines an affine cone in the tangent space $T_{X,q}$ to $X$ at $q$. Since the tangent space $T_{X,q}$ is an affine open subscheme of $\PPT_{X,q}$, the closure of this affine cone gives rise to a projective cone $Q_{f_2}$ in $\PPT_{X,q}$. The projective cone $Q_{f_2}$ in the projective space $\PPT_{X,q}$ is defined by $f_2$ as well.

On the other hand, since the general fiber $\F$ has positive dimension, the evalution map \[ev:\M_{0,2}(X,2)\rightarrow X\times X\] is surjective. Therefore, any two points of $X$ can be connected by a conic. Let $C_q$ be the space of conics contained in $X$ and passing through $q$. The projective tagnent lines to the conics parametrized by $C_q$ at the point $q$ sweep out $\PPT_{X,q}$. Suppose that $C$ is a conic parametrized by a general point of $C_q$. We conclude that the projective tangent line $\PPT_{C,q}$ to $C$ at $q$ meets the cone $Q_{f_2} (\subseteq \PPT_{X,q})$ only at the point $q$. In other words, we have that\[q=\PPT_{C,q}\cap Q_{f_2}.\]By the surjectivity of the map $ev$, we can assume that one of such general conics passes through $p$ and denote by $C$ as well.

We consider the intersection $\Span(C)\cap \PPT_{X,q}$. It is the projective tangent line $\PPT_{C,q}$. Since the intersection $\PPT_{C,q}\cap \PPT_{X,p}$ is a point on $\PPT_{C,q}$ distinct from $q (=\PPT_{C,q}\cap Q_{f_2})$, the intersection\[\Span(C) \cap \PP^{n-2}= \Span(C) \cap \PPT_{X,q} \cap \PPT_{X,p}=\PPT_{C,q}\cap \PPT_{X,p}\] is not contained in the projective cone $Q_{f_2}$. On the other hand, by Lemma \ref{var}, we know that $\Span(C) \cap \PP^{n-2}$ coincides with $\varphi([C])$ where $[C]\in \F$ is the point parametrizing $C$. Therefore, we conclude that $f_2$ does not vanish at the point $[C]\in \F$. We show the theorem when $X$ is a smooth hypersurface. In general, for a smooth complete intersection $X$, the proof is similar.

%For a general point $p$, the map $\rho$ is mapping a point $(C,p,q)\in \F$ to a point
%\begin{equation}\label{eq:map1}
%\rho([C,p,q])=\PP T_{C,p} \cap \PP T_{X,q}\in \PP T_{X,p}\cap \PP
 %T_{X,q}=\PP^{n-2}.
 %\end{equation}

%In particular, the has a in $\PP T_{X,q}$ which is not on the cone $Q$. Therefore, for a general point $p(\neq q)$ on the conic $C$, the tangent line $\PP T_{C,p}$ does not intersect the cone $Q$. From the view of point of the map \eqref{eq:map1}, it means that $f_2$ does not vanish on $\rho(\F)$.
\end{proof}
  \textbf{Acknowledgments.} The author is very grateful for his advisor  Professor~de Jong for suggesting this project. The author thanks  Professor~Liu and  Professor~Fedorchuk for pointing out some references. The author also thanks Professor~Starr, Zhiyu Tian and Yi Zhu for the encouragement. In the end, the author gratefully acknowledge the anonymous referee for reading this paper carefully and providing a lot of useful suggestions.


\begin{thebibliography}{0}
\bibitem[ACG11]{GH}
Enrico Arbarello, Maurizio Cornalba, and Pillip~A. Griffiths.
\newblock {\em Geometry of algebraic curves. {V}olume {II}}, volume 268 of {\em
  Grundlehren der Mathematischen Wissenschaften [Fundamental Principles of
  Mathematical Sciences]}.
\newblock Springer, Heidelberg, 2011.
\newblock With a contribution by Joseph Daniel Harris.

\bibitem[AM69]{A}
M.~F. Atiyah and I.~G. Macdonald.
\newblock {\em Introduction to commutative algebra}.
\newblock Addison-Wesley Publishing Co., Reading, Mass.-London-Don Mills, Ont.,
  1969.

\bibitem[Bea95]{BV}
Arnaud Beauville.
\newblock Quantum cohomology of complete intersections.
\newblock {\em Mat. Fiz. Anal. Geom.}, 2(3-4):384--398, 1995.



\bibitem[CS09]{Cubic}
Coskun, Izzet and Starr, Jason.
\newblock {Rational curves on smooth cubic hypersurfaces}.
\newblock Int. Math. Res. Not. IMRN (24)4626--4641, 2009

\bibitem[CZ15]{ZC}
Qile Chen and Yi~Zhu.
\newblock Strong approximation over function fields.
\newblock {\em http://arxiv.org/abs/1510.04647 }  Preprint 2015.

\bibitem[dJHS11]{DS2}
A.~J. de~Jong, Xuhua He, and Jason~Michael Starr.
\newblock Families of rationally simply connected varieties over surfaces and
  torsors for semisimple groups.
\newblock {\em Publ. Math. Inst. Hautes \'Etudes Sci.}, (114):1--85, 2011.

\bibitem[dJS06]{DS}
A.~J. de~Jong and Jason~Michael Starr.
\newblock Low degree complete intersections are rationally simply connected.
\newblock {\em Preprint}, 2006.

\bibitem[FP97]{FP}
W.~Fulton and R.~Pandharipande.
\newblock Notes on stable maps and quantum cohomology.
\newblock In {\em Algebraic geometry---{S}anta {C}ruz 1995}, volume~62 of {\em
  Proc. Sympos. Pure Math.}, pages 45--96. Amer. Math. Soc., Providence, RI,
  1997.

\bibitem[Ful98]{F}
William Fulton.
\newblock {\em Intersection theory}, volume~2 of {\em Ergebnisse der Mathematik
  und ihrer Grenzgebiete. 3. Folge. A Series of Modern Surveys in Mathematics
  [Results in Mathematics and Related Areas. 3rd Series. A Series of Modern
  Surveys in Mathematics]}.
\newblock Springer-Verlag, Berlin, second edition, 1998.

\bibitem[Har66]{H2}
Robin Hartshorne.
\newblock {\em Residues and duality}.
\newblock Lecture notes of a seminar on the work of A. Grothendieck, given at
  Harvard 1963/64. With an appendix by P. Deligne. Lecture Notes in
  Mathematics, No. 20. Springer-Verlag, Berlin, 1966.

\bibitem[Har77]{H}
Robin Hartshorne.
\newblock {\em Algebraic geometry}.
\newblock Springer-Verlag, New York, 1977.
\newblock Graduate Texts in Mathematics, No. 52.

\bibitem[CaK02]{CaJ}
Carolina Araujo and J\'anos Koll\'ar
\newblock {\em Rational Curves on Varieties}.
\newblock Arxiv, preprint.

\bibitem[Has10]{BH}
Brendan Hassett.
\newblock Weak approximation and rationally connected varieties over function
  fields of curves.
\newblock In {\em Vari\'et\'es rationnellement connexes: aspects
  g\'eom\'etriques et arithm\'etiques}, volume~31 of {\em Panor. Synth\`eses},
  pages 115--153. Soc. Math. France, Paris, 2010.

\bibitem[Kol96]{K}
J{\'a}nos Koll{\'a}r.
\newblock {\em Rational curves on algebraic varieties}, volume~32 of {\em
  Ergebnisse der Mathematik und ihrer Grenzgebiete. 3. Folge. A Series of
  Modern Surveys in Mathematics [Results in Mathematics and Related Areas. 3rd
  Series. A Series of Modern Surveys in Mathematics]}.
\newblock Springer-Verlag, Berlin, 1996.

\bibitem[Mat89]{Ma}
Hideyuki Matsumura.
\newblock {\em Commutative ring theory}, volume~8 of {\em Cambridge Studies in
  Advanced Mathematics}.
\newblock Cambridge University Press, Cambridge, second edition, 1989.
\newblock Translated from the Japanese by M. Reid.

\bibitem[Mum77]{D}
David Mumford.
\newblock {\em Stability of projective varieties}.
\newblock L'Enseignement Math\'ematique, Geneva, 1977.
\newblock Lectures given at the ``Institut des Hautes {\'E}tudes
  Scientifiques'', Bures-sur-Yvette, March-April 1976, Monographie de
  l'Enseignement Math{\'e}matique, No. 24.





\bibitem[Mum08]{AVB}
David Mumford.
\newblock {\em Abelian varieties}, volume~5 of {\em Tata Institute of
  Fundamental Research Studies in Mathematics}.
\newblock Published for the Tata Institute of Fundamental Research, Bombay,
  2008.
\newblock With appendices by C. P. Ramanujam and Yuri Manin, Corrected reprint
  of the second (1974) edition.

% \bibitem[CLH08]{Conics}
%C\'esar Lozano Huerta.
%\newblock {\em The Hilbert scheme seminar}, Preprint.
\bibitem[Ser06]{deform}
Edoardo Sernesi.
\newblock {\em Deformations of algebraic schemes}, volume 334 of {\em
  Grundlehren der Mathematischen Wissenschaften [Fundamental Principles of
  Mathematical Sciences]}.
\newblock Springer-Verlag, Berlin, 2006.




\end{thebibliography}
\end{document}